\documentclass{article}
\usepackage{fancyhdr}
\usepackage{tikz,amsmath,amsthm,amssymb,amsfonts,graphicx,dsfont,algorithmic}
\usepgflibrary{decorations.pathmorphing}
\usetikzlibrary{decorations.pathmorphing}
 \usepackage{caption}
 \usepackage{subfigure}
 \usepackage{authblk}
\usepackage{colortbl,color,hhline}
\usepackage{pst-all}
\usepackage{subfigure}\usepackage{mathtools}

\DeclarePairedDelimiter\floor{\lfloor}{\rfloor}

\newenvironment{proofs}{\paragraph{Proof Sketch:}}{\hfill\null}
\newtheorem{thm}{Theorem}[section]

\newtheorem{lem}[thm]{Lemma}

\newtheorem{defn}[thm]{Definition}

\title{A proof for Padberg's conjecture on rank of matching polytope}
\author[]{Ashwin Arulselvan\thanks{arulsel@math.tu-berlin.de} } 
\author[]{Daniel Karch\thanks{karch@math.tu-berlin.de}}

\affil[]{\small Institut f\"ur Mathematik\\Technische Universit\"at Berlin\\ 10623 Berlin, Germany}

\begin{document}
\date{}
\maketitle

\begin{abstract}
Padberg~\cite{Pa:13} introduced a geometric notion of ranks for (mixed) integer rational polyhedrons and conjectured that the geometric rank of the matching polytope is one. In this work, we prove that this conjecture is true.
\end{abstract}

\section{Preliminaries}
Padberg~\cite{Pa:13} defined the notion of geometric ranks for the facets of the polytopes. All definitions provided in this section could be found in~\cite{Pa:13}.
Let the following be the set of feasible solutions to an integer linear program.
\[Q = \{x \in \mathbb{Z}^n: Ax\le b\}\]  
We can describe $P=conv(Q)$ (convex hull of $Q$) by the solution set of \textit{ideal}, i.e. a minimal linear description as follows.
\[P = \{x \in \mathbb{R}^n: H_1x= h_1, H_2x\le h_2\},\]  
where $(H_1, h_1)$ is $s\times (n+1)$ full rank matrix. Thus the dimension of $P$ is $dim (P) = n - s$. $(H_2,h_2)$ is $t \times(n+1)$ matrix of rationals for some integer $t$. The rows of the matrix defines a facet of $P$.
Let $\mathcal{F}$ be the set of row vectors $(h, h_0)\in\mathbb{R}^{n+1}$ from the matrix $(H_2,h_2)$. Each of these inequalities induces a facet of $P$. For a $(f,f_0)\in\mathcal{F}$, the facet
\[F_f = \{x\in P: fx=f_0\}\]
is a polyhedron of dimension $dim(P)-1$. In order to define the notion of ranks, Padberg~\cite{Pa:13} required the notion of facet of facets, i.e., the ridges of $P$. Let
\[\mathcal{H}^f = \{(h,h_0)\in \mathcal{F}: dim(F_f\cap F_h) = dim(P)-2\}\]
A facet defining inequality $hx\le h_0$ of $P$ is qualified for membership in set $\mathcal{H}^f$ if and only if it defines a facet of $F_f$ (ridge of $P$). If $\mathcal{F}$ is ideal for $P$, for a subset $\mathcal{F}_{\min}\subseteq \mathcal{F}$, we define
\[P_{\min} = \{x\in \mathbb{R}^n: H_1x=h_1, fx\le f_0, \forall (f,f_0)\in \mathcal{F}_{\min}\}\]
\begin{defn}{}
We call a subset $\mathcal{F}_{\min}$ as a minimal formulation for  $P$, if it satisfies
\begin{itemize}
\item[(i)] $P = conv(P_{\min}\cap\mathbb{Z}^n)$
\item[(ii)] $P\cap\mathbb{Z}^n \subset P_{\min}^h \cap\mathbb{Z}^n, \forall (h,h_0) \in \mathcal{F}_{\min}$,
\end{itemize}
where
\[P_{\min}^h = \{x\in \mathbb{R}^n: H_1x=h_1, fx\le f_0, \forall (f,f_0)\in \mathcal{F}_{\min} - (h,h_0)\}\] and the containment in (ii) is proper.
\end{defn}

Rank 0 facets of a polytope $P$ are given by 
\[\mathcal{F}_0 = \{(f,f_0)\in \mathcal{F}: \text{if $(f,f_0)$ is in some minimal formulation of $P$}\}\]
We now build a hierarchy of facets as follows: Given a family of non-empty disjoint family of facets of $P$, $\mathcal{F}_0, \mathcal{F}_1,\dots, \mathcal{F}_{\rho}$, such that $\bigcup_{i=0}^{\rho}\mathcal{F}_i \neq \mathcal{F}$, we define
\[\mathcal{F}_{\rho+1} = \{(f,f_0)\in \mathcal{F}-\bigcup_{i=0}^{\rho}\mathcal{F}_i: \exists (h,h_0)\in \bigcup_{i=0}^{\rho}\mathcal{F}_i, \text{ s.t. } dim(F_f\cap F_h) = dim(P)-2\}\]

\begin{defn}
If $\mathcal{F}= \bigcup_{i=0}^{\rho}\mathcal{F}_i$, then we call $\rho(P)=\rho$ as the geometric rank of the polyhedron $P$.
\end{defn}

Padberg~\cite{Pa:13} has shown that the geometric number is finite and well defined number for every polyhedron.

\section{Properties of factor critical graphs}

We discuss a few facts about factor critical graphs and 2-connected graphs (node connected) that we will be requiring in our proofs. Much of these could be found in~\cite{LoPl:86}.

\begin{defn}{}
A graph $G(V,E)$ is a ``factor critical graph'' if for every node $v\in V, G- \{v\}$ has a perfect matching (p.m).
\end{defn}

\begin{defn}{}
For a subset of nodes $U'\subset V$, the node induced subgraph $G[U']$ of $G$ is a ``nice subgraph'' of $G$ if $G-V(G[U'])$ has a p.m.
\end{defn}

We sketch the simple proof of the following lemma, as we will be using this idea in our proof as well.
\begin{lem}{\cite{LoPl:86}}\label{lem:oc}
If $G$ is a factor citical graph, then every edge of $G$ lies in some nice odd cycle.
\end{lem}
\begin{proof}
For an edge $e=(i,j)$, find a p.m in $G-\{i\}$ and call it $M_i$ and find a p.m $M_j$ in $G-\{j\}$. There is an alternating path $P$ (edges of the path alternating between the edges of $M_i$ and $M_j$) between $i$ and $j$ in $M_i\cup M_j$ that is necessarily even. Now, $P+(i,j)$ gives us the desired nice odd cycle.
\end{proof}

Let $G$ be a graph and $G'$ be its subgraph. Then an \textbf{ear} of $G$ relative to $G'$ is a path with both end nodes (but none of the interior nodes) in $G'$. A ear is \textbf{proper} if the endpoints are distinct. If the ear is a cycle, then it is not proper. The ear is an \textbf{odd ear}, if it has an odd length. A ear decomposition of $G$ starting from its subgraph $G[U']$ is its representation given by $G=G'+P_1+P_2+\dots+P_k$, for some $k$. $P_i$ is an ear of $G'+P_1+\dots+P_i$ relative to graph $G'+P_1+\dots+P_{i-1}, i\le k$. A proper odd ear decomposition is one in which every ear is odd and proper.

\begin{thm}{\cite{LoPl:86}}\label{thm:fcNC}
Let $G$ be a graph and $C$ be a nice odd cycle of $G$. Then $G$ is factor critical if and only if there is an odd ear decomposition of $G$ starting from $C$, $G=C+P_1+P_2+\dots+P_k$. Moreover, $G=C+P_1+P_2+\dots+P_i$, for all $i\le k$, is a nice factor critical subgraph of $G$.
\end{thm}

We will also be making use of the the following obvious fact: Any ear decomposition of a nice factor critical subgraph of $G$ extends to an ear decomposition of $G$~\cite{LoPl:86}.

\begin{thm}{\cite{LoPl:86}}\label{thm:fc2}
Let $G$ be a factor critical and 2-connected graph of $G$. Then $G$ has a proper odd ear decomposition $G = C+P_1+P_2+\dots+P_k$ for some $k$ and $C$ is some nice odd cycle of $G$. Moreover, $G_i = C+P_1+P_2+\dots+P_i, i\le k$ is a factor critical, 2-connected nice subgraph of $G$.
\end{thm}

\section{Matching polytope and Padberg's conjecture}
 
For a graph $G(V,E)$, let $\delta(v)$ be the set of edges incident on a node $v\in V$ in $G$ and let $E[U]$ be the edge set of the node induced subgraph $G[U]$ for some subset of nodes $U\subset V$. The facets of the matching polytope for graph $G=(V,E)$ is given by the following set of inequalities.
\begin{itemize}
\item[(i)] $x_e \ge 0, \forall e \in E$
\item[(ii)] $\sum_{e\in \delta(v)]}x_e \le 1, \forall v\in V, \delta(v) \ge 3$ or $\delta(v) = 2$ and not a part of triangle
\item[(iii)] $\sum_{e\in E[U]}x_e \le \floor[\Big]{\frac{1}{2}|U|}, \forall U \subseteq V, G[U]$ is factor critical and 2-node connected
\end{itemize}

The variable $x_e$ takes a value of 1 if edge $e$ is picked in the matching and takes 0 otherwise. The matching polytope is full dimensional. The reader could refer to~\cite{Sc:03} for more details on facets and dimension of the matching polytope.

All efforts in this work is dedicated to prove the following theorem that was stated as a conjecture by Padberg\cite{Pa:13} about the matching polytope.
\begin{thm}~\label{thm:main}
The geometric rank of the matching polytope is less than or equal to 1.
\end{thm}
\textit{Note: A cycle with four nodes has a rank 0, every facet is required in the minimal formulation}

We first need the following result to prove theorem~\ref{thm:main}.

\begin{lem}\label{lem:rank}
%The facets with rank $\ge 1$ correspond to facets induced by inequalities of type (iii).
Facets induced by inequalities of type (i),  (ii) and the triangles corresponding to inequality (iii) with a node of degree 2, together constitute a minimal formulation and hence has a rank 0.
\end{lem}
\begin{proof}
In the presence of the integrality constraint, we are restricted to binary values and clearly every valid matching $\tilde{x}$ is satisfied by these inequalities. In order to show it is minimal, we need to argue why none of them could be removed from the formulation.
%Follows from the simple observation that constraints of type (i) and (ii) belongs to a minimal formulation. By just having inequalities of type (i) and (ii) in the formulation and retaining the integrality constraints, we could see that we cannot throw away any of the inequalities to obtain a formulation with one less inequality, making it minimal. 
We can pick negative values without type (i). We can pick any two edges incident on the same node without type (ii). We can pick the two edges incident on the node of degree 2 in the triangle corresponding to the type (iii) inequality, without this inequality. 
\end{proof}

%\begin{lem}\label{lem:tri}
%Inequalities corresponding to triangles containing a node of degree 2 has rank 0
%\end{lem}
%\begin{proof}
%We could easily see that a formulation without these inequalities, we could pick the two edges that are incident to this node with degree two in a solution, which is not feasible to a matching. 
%\end{proof}

\begin{proof}[Proof of the Theorem~\ref{thm:main}]
We will prove the claim by showing that every type (iii) inequality (not corresponding to a triangle with a node of degree 2) induces a facet to the facet induced by some inequality of type (ii). This along with lemma~\ref{lem:rank} will prove the claim.
We know for a fact that a component $U$ that is involved in type (iii) inequality is factor critical and 2-node connected. 
%So by theorem~\ref{thm:fc2}. $G[U]$ has proper odd ear decomposition $G[U] =  C+P_1+P_2+\dots+P_k$. We choose the point $n$, where the subgraph $G_n=C+P_1+\dots+P_n$ starts spanning of $G$. So all ears $P_{n+1},\dots,P_k$ are single edges. Note that $P_n$ has at least two interior vertices that is not present in $G_{n-1}$. We will now choose an internal vertex $v$ of the ear $P_n$. Also, note that $G_n$ is a factor critical 2-connected graph (by theorem~\ref{thm:fc2}) and any near p.m in $G_n$ is a near p.m to $G$ as well.
So it has proper odd ear decomposition. We will pick a node to which an ear is connected. Let us call it $v$. The degree inequality corresponding to this node would serve as our type (ii) inequality candidate. Since an ear is connected to $v$, it has a degree of at least 3 in $G[U]$. If we don't find such a node, then $G[U]$ is an odd cycle without any chords (odd hole) and it is easy to show that it induces a facet to the facet induced by the degree constraints (a type (ii) inequality) of any of the nodes in the hole (we will deal with this case separately). Now let us consider the two inequalities corresponding to component $U$ and node $v$.
\begin{align}
\sum_{e\in E[U]}x_e &\le \floor[\Big]{\frac{1}{2}|U|}\label{eq:tp1}\\
\sum_{e\in \delta(v)}x_e &\le 1\label{eq:tp2}
\end{align}
Let 
\[P^2 = \{x\in \mathbb{R}^{|E|}: Ax \le b, \sum_{e\in E[u]}x_e = \floor[\Big]{\frac{1}{2}|U|}, \sum_{e\in \delta(v)]}x_e = 1\},\]
where $Ax\le b$ is the system with all the inequalities (i), (ii) and (iii) except~\ref{eq:tp1}~and~\ref{eq:tp2}. If $dim(P^2) < dim(P) -2$, then we know that there is some inequality in $Ax \le b$ that is implicitly satisfied at equality. For every inequality $\alpha x\le \beta$ in the system $Ax\le b$, we will show that there exists a valid matching $\tilde{x}$ satisfying~\ref{eq:tp1}~and~\ref{eq:tp2} at equality and we have $\alpha\tilde{x}< \beta$ true. In other words, we would have shown that none of the inequalities in the system $Ax\le b$ is implicitly satisfied at equality in $P^2$.\\

\textbf{Case 1:} Type (i): $x_e \ge 0$:\\

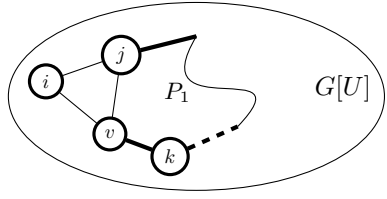
\begin{figure}
\begin{center}
\begin{tikzpicture} [>=stealth]

	\node (i) at (-1,0.65)   [scale=0.75,circle,draw,very thick] {$i$};	
	\node(j) at (0,1)   [scale=0.75,circle,draw,very thick] {$j$};
	\node (v) at (-0.15,-0.05)   [scale=0.75,circle,draw,very thick] {$v$};
	\node (k) at (0.65,-0.35)   [scale=0.75,circle,draw,very thick] {$k$};
	\draw (1,0.45) ellipse (2.5cm and 1.25cm);
	\draw(i)--(j);
	\draw(v)--(j);
	\draw(i)--(v);
	\draw[ultra thick](j)--(1,1.25);
	\draw[ultra thick](k)--(v);
	\draw[ultra thick, dashed](k)--(1.55,0.05);
	\draw (1,1.25) .. controls (0.25,0) and (2.5,1.05) .. (1.55,0.05);
	\node (p1) at (0.75,0.5) [draw=none,fill=none] {\small{\text{$P_1$}}}; 
	\node (Gu) at (2.95,0.55) [draw=none,fill=none] {{\text{$G[U]$}}}; 
\end{tikzpicture}

\caption{Cycle $C$ is a triangle with nodes $i,j, v$}\label{fig:case1b}
\end{center}
\end{figure}

\textit{Case 1a:} $e\notin E[U]$. It is easy to see that we can set this edge to a value 1, to get a valid matching that satisfies both inequalities~\ref{eq:tp1}~and~\ref{eq:tp2} at equality. If one of the endpoints is in $U$, we work with the p.m obtained by deleting this endpoint. This allows us to pick this edge into the matching.\\

\textit{Case 1b:} $e=(i,j) \in E[U]$. We know from lemma~\ref{lem:oc} that every edge is present in a nice odd cycle in $G[U]$. We can get the perfect matching without this cycle and within this odd cycle we can get a near perfect matching that picks this edge. We need to deal with the case if the nice odd cycle is a triangle with nodes, $i,j$ and $v$. As this would force us to delete $v$ in the near perfect matching and we might not be able to satisfy (2) at equality. In order to handle this case, we get an odd ear decomposition starting from $C$, $G[U]=C+P_1+P_2+\dots+P_k$, similar to the one provided in~\cite{LoPl:86}. The first ear we construct contains the edge $(v,k)$, where  $k\in G[U] - C$. The first ear $P_1$ is constructed from $M_v\cup M_k$. This is an alternating path connecting  $v$ and $k$ and we add edge $(v,k)$ to it to obtain the first ear $P_1$. Note that $C+P_1$ is a nice subgraph of $G$. If $P_1$ is a proper ear with endpoints $v$ and say $j$ (see figure~\ref{fig:case1b}), then the $P_1+j\text{ --- }i\text{ --- }v$ forms another nice odd cycle of $G$ and we reduce it to the previous case. If $P_1$ is not a proper ear, then it is a odd cycle containing $v$. We just get near p.m in $P_1$ that contains $v$ and pick edge $(i,j)$ and the p.m in $G-C-P_1$. In all cases, we got a near perfect matching that has an edge incident to $v$ and the edge $(i,j)$ was picked.\\

% Let $e' \neq (i,v), (j,v)$ be the edge incident on $v$ and present in $G[U]$. Such an edge exists, since we assumed $v$ has a degree at least 3 in $G[U]$. The first proper odd ear in the decomposition would contain this edge. We first find the near perfect matching containing this edges $M_{e'}$. If this contains $e=(i,j)$, we have the desired matching. If not, we construct the first proper odd ear $P_1$ (to be added to $C$) by choosing the alternating path obtained by the near perfect matching $M_{e'}$ and the perfect matching $M_C$ obtained in the graph $G\backslash C$. This path $P_1$ and $C$ would give a new nice odd cycle of $G$ containing edge $(i,j)$ and node $v$ with more than 3 edges (see figure~\ref{fig:case1b}).\\

\textbf{Case 2:} Type (ii): For some vertex $v'\ne v, \sum_{e\in \delta(v')}x_e \le 1$:\\
This is a relatively simple case to show. If $v'\in U$, we will just delete this node and get a matching in $G[U]$, and by not including any of the other edges incident to $v'$, we get a matching that doesn't satisfy this at equality. The case, $v'\notin U$, is easy to see.\\

\textbf{Case 3:} Type (iii): For some vertex set $U'\ne U, \sum_{e\in E[U']}x_e \le \floor[\Big]{\frac{1}{2}|U'|}$:\\
If we can show that there is a matching $\tilde{x}$ satisfying~\ref{eq:tp1}~and~\ref{eq:tp2} at equality with at least two nodes in $U'$ not matched with nodes in $U'$ (either unmatched or matched with nodes in $V-U'$) and since $U'$ is odd, we would have that $\sum_{e\in E[U']}\tilde{x}_e < \floor[\Big]{\frac{1}{2}|U'|}$.\\

\textit{Case 3a:} Either $|U'| \ge |U|$ or  $U' \cap U = \emptyset$. Deleting any node other than $v$ in $U$ to get a near p.m in $G[U]$ would yield the desired matching for this case. As there are at least two nodes in $U'$ (both $U'$ and $U$ are odd) that is not matched.\\

\textit{Case 3b:} $U' \cap U \ne \emptyset, U' \not\subset U$. We now delete a node $w\in U' \cap U\backslash\{v\}$ and get a near p.m for $G[U]$. So we know that we have a matching that does not match $w\in U'$ and some other node in $w'\in U'\backslash U$. We can just choose not to match $w'$ to obtain a matching, for which the inequality is not satisfied at equality. If $v$ is the only node in the intersection, then we delete some other node in $U$ and find a near p.m for $G[U]$. Node $v$ will be matched to some node in $U$. So we will have node $v$ and $w'$ in $U'$ not matched within $U'$ in this solution.\\

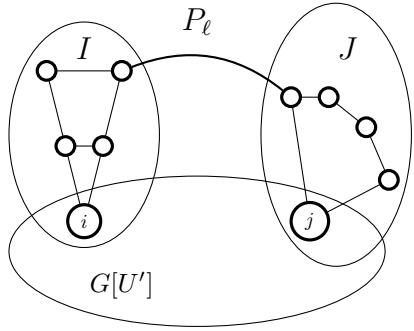
\begin{figure}
\begin{center}
        \begin{tikzpicture} [>=stealth]
	\node (i) at (-0.5,0.85)   [scale=0.75,circle,draw,very thick] {$i$};	
	\node (j) at (2.5,0.85)   [scale=0.75,circle,draw,very thick] {$j$};	
%	\node (i') at (0.5,4.5)   [scale=0.85,circle,draw,very thick] {$i'$};	
%	\node (v) at (1,0)   [scale=0.95,circle,draw,very thick] {$v$};	
%	\node (j') at (1.5,4.5)   [scale=0.75,circle,draw,very thick] {$j'$};	
	\node (1) at (-0.75,1.85)   [scale=0.75,circle,draw,very thick] {};	
	\node (2) at (-0.25,1.85)   [scale=0.75,circle,draw,very thick] {};	
	\node (3) at (-1,2.85)   [scale=0.75,circle,draw,very thick] {};	
	\node (4) at (0,2.85)   [scale=0.75,circle,draw,very thick] {};	

	\node (5) at (2.25,2.5)   [scale=0.75,circle,draw,very thick] {};	
	\node (6) at (2.75,2.5)   [scale=0.75,circle,draw,very thick] {};	
	\node (7) at (3.25,2.1)   [scale=0.75,circle,draw,very thick] {};	
	\node (8) at (3.55,1.4)   [scale=0.75,circle,draw,very thick] {};

	\draw (1,0.45) ellipse (2.5cm and 1cm);
	\draw (-0.5,2) ellipse (1cm and 1.5cm);
	\draw (2.85,2) ellipse (1cm and 1.75cm);
	\draw (1)--(2);
	\draw (1)--(i);
	\draw (i)--(2);
	\draw (1)--(3);
	\draw (4)--(3);
	\draw (4)--(2);
	\draw (5)--(j);
	\draw (8)--(j);
	\draw (5)--(6);
	\draw (7)--(6);
	\draw (7)--(8);

	\draw[thick] (4) edge [bend left] (5);	
	
 	\node (0,-2)[draw=none,fill=none] {{\text{$G[U']$}}}; 
 	\node (pell) at  (1,3.5) [draw=none,fill=none] {\large\textbf{\text{$P_\ell$}}}; 
 	\node (I) at  (-0.5,3.15) [draw=none,fill=none] {\large\textbf{\text{$I$}}}; 
 	\node (J) at  (3,3.15) [draw=none,fill=none] {\large\textbf{\text{$J$}}}; 
	
	%\draw[ultra thick](i') -- (-0.25,4.5);
%	\draw[ultra thick](i') -- (-0.2,2.85);
%	\draw[ultra thick,dashed](j') -- (2.5,4.5);
%	
%	\draw[ultra thick](i) -- (-0.75,1.75);
%	\draw[ultra thick, dashed](-0.1,2.25) -- (-0.75,1.75);
%	\draw[ultra thick](-0.1,2.25) -- (-0.75,2.5);
%	\draw[ultra thick,dashed](-0.2,2.85) -- (-0.75,2.5);
%	%\draw[ultra thick](-0.2,2.85) -- (-1.5,3.25);
%	
%	\draw[ultra thick, dashed](j) -- (2.85,1.75);
%	\draw[ultra thick](2.25,2) -- (2.85,1.75);
%	\draw[ultra thick,dashed](2.25,2) -- (2.85,2.25);
%	\draw[ultra thick](2.25,2.75) -- (2.85,2.25);
%	\draw[ultra thick, dashed](2.25,2.75) -- (3.35,3.35);
%	\draw[decoration = {zigzag,segment length = 2mm, amplitude = 0.5mm}, decorate](2.5,4.5) -- (3.35,3.35);
%	%\draw[decoration = {zigzag,segment length = 2mm, amplitude = 0.5mm}, decorate](-1.5,3.25) -- (-0.25,4.5);
%	\draw[-] (i) -- (v);	
%	\draw[-] (j) -- (v);	
	%\draw[ultra thick](i) -- (-0.75,1.75);
    	%\draw [ultra thick](i) -- (x11)   ;
		
        \end{tikzpicture}

\end{center}
\caption{Construction of graph $G_\ell$ from $G[U']$. Components $I$ and $J$ are attached to $G[U']$ with the last ear connecting $I$ and $J$ to obtain $G_\ell \subseteq G[U]$.}\label{fig:gellconst}
\end{figure}

\textit{Case 3c:} $U'\subset U$. 
It is easy to see the case when $G[U']$ is not a nice subgraph of $G[U]$. We claim that deleting any node $w$ in $U'$ and doing a matching on $U$ would result in at least two of the nodes in $U'$ to match with nodes in $U\backslash U'$. If every node of $U'\backslash \{w\}$, in our p.m. $M_w$, was matched with a node in $U'$, then every node in $U\backslash U'$ is matched with a node in $U\backslash U'$. This would mean $G[U']$ is a nice subgraph of $G[U]$. So at least one node in $U'\backslash\{w\}$ is matched with a node in $U\backslash U'$. Since $U'$ has an odd number of nodes, there should be one other node in $U'$ that must be matched outside $U'$ (to some node in $U\backslash U'$) in $M_w$.

For the case that $G[U']$ is a nice subgraph of $G[U]$, since $G[U']$ is a factor critical graph, we know that it has an odd ear decomposition. $G[U']=C+P_1+P_2+\dots+P_k$. Since we are dealing with the case that $G[U']$ is a nice factor critical subgraph of $G[U]$ the ear decomposition of $G[U']$ extends to $G[U]$~\cite{LoPl:86}. So we have, $G[U]=C+P_1+P_2+\dots+P_k+P'_1+P'_2+\dots+P'_{k'}$. We know that $G[U]$ is 2-connected. We also know that $G[U']$ does not span $G[U]$, i.e., $U'\backslash U$ is not empty. This implies that we have at least two nodes $i$ and $j$ in $G[U']$ that have distinct neighbours in $G[U]-G[U']$. Now as we are adding our ``extended'' ears ($P'_1,\dots, P'_{k'}$), we are growing the components $I$ and $J$ that are attached to $i$ and $j$ respectively. $I$ (resp. $J$) is the set of nodes that gets disconnected from $G[U]+I+J$ (see figure~\ref{fig:gellconst}) if we delete $i$ (resp. $j$). We stop at the point, where we add an ear and these two components get connected, and we call this $\ell$. We get a new graph $G_\ell = C+P_1+P_2+\dots+P_k+P'_1+P'_2+\dots+P'_\ell$ (see figure~\ref{fig:gellconst}), which is a nice factor critical subgraph of $G[U]$.
Note that there could be more than two nodes in $G[U]$ that have neighbours in $G[U]-G[U']$. We choose the first two nodes as $i$ and $j$ for which we could stop at $\ell$.  Let us say $P^i_1\dots P^i_{k_i}$ (resp. $P^j_1\dots P^j_{k_j}$) are the $k_i$ (resp. $k_j$) ears constituting component $I$ (resp. $J$) that is attached to $i$ (resp. $j$). We can assume that we haven't yet added the ears constituting the components attached to the nodes that are not $i$ and $j$ of $G[U']$. In other words, $G_\ell= G[U'] + P^i_1+\dots+P^i_{k_i}+P^j_1+\dots+P^j_{k_j}+P_\ell$.

If the ear $P_\ell$ is a single edge we refer to its endpoints as $i'$ and $j'$, else we pick the first two adjacent interior nodes of this ear ($i'$ is the closest node to component $I$) and refer to them as $i'$ and $j'$. We are now going to construct a nice odd cycle, $C$, for $G_\ell$, that contains this edge $(i', j')$ (see figure~\ref{fig:case3ca}). The construction is similar to the proof of lemma~\ref{lem:oc}. We get a p.m, $M_{i'}$ in $G_\ell-\{i'\}$ and $M_{j'}$ in $G_\ell-\{j'\}$ and add $(i',j')$ to the alternating path connecting $i'$ and $j'$ in $M_{i'}\cup M_{j'}$. A few key observations are in order. Note that we address the components attached to $i $ and $j$ as $I$ and $J$ before we added $P_\ell$. The number of nodes in both components $I$ and $J$ (respectively excluding $i$ and $j$) is even, since we only add odd ears. Now the near p.m $M_{j'}$ would force us to pick an edge $(i,t)$ for some $t\in I$, since $I\cup\{i'\}$ has an odd number of nodes (excluding $i$) and one node in $I\cup\{i'\}$ is forced to match with $i$.
%the component attached to $i$ after the deleting of $j'$ in the ear $P_\ell$ has odd number of nodes (excluding $j$) and hence one node in that component has to be matched with $j$. 
By a similar argument $M_{i'}$ will pick an edge $(j,t')$ for some $t'\in J$. So the path $i\text{ --- }i'\text{ --- }j'\text{ --- }j$ in our constructed odd cycle is of odd length with at least two nodes. This is true even when $P_\ell$ is a single edge ear or when either $i$ and $i'$ or $j$ and $j'$ coincide by similar parity arguments. The other even path connecting $i$ and $j$ of our nice cycle will be entirely inside $G[U']$ due to the way we constructed $G_\ell$. It contains at least one node, other than $i$ and $j$,  because both $i$ and $j$ are matched to some nodes $k'$ and $k$ outside $G[U']$ in the matchings $M_{j'}$ and $M_{i'}$ respectively. In other words, either matchings would not have picked the edge $(i,j)$, even if it existed. Notice that, since $C$ is a nice subgraph of $G_\ell$ and $G_\ell$ is a nice subgraph of $G[U]$, $C$ is a nice subgraph of $G[U]$.
%The matching that we now construct will first get  the p.m in $G-C$. If $v$ is either not in $G[U']$ or not a part of the cycle $C$, then we just get the p.m in $G-C$ and delete some node, say $m$, in the even path between $i-j$, that is adjacent to either $i$ or $j$ to get a near p.m in $C$. This near p.m would force us to match $i$ and $j$ to be matched with nodes outside $G[U']$. Thus we obtained a near p.m in $G$ with atleast two nodes in $G[U']$ (node $m$ and either $i$ or $j$) unmatched within $G[U']$.
In the even path $i\text{ --- }j$, if $v$ is not the only node adjacent to $i$ and $j$ (like in figure~\ref{fig:case3cc}), then we delete the node, say $m$, adjacent to either $i$ or $j$ and $m\neq v$ and get a p.m in $C$ with $m$ deleted and a p.m for $G[U]-C$. The matching constructed would have $i,j$ and $m$ either unmatched or matched to nodes outside $G[U']$. If either $i$ or $j$ is node $v$, then it is easy to see that we could obtain a matching with at least two nodes in $U'$ matched outside $G[U']$ and an edge incident on $v$ being picked.

If $v$ is the only node adjacent to $i$ and $j$ (see figure~\ref{fig:case3cc}), we know from theorem~\ref{thm:fcNC} that $G[U]$ has a ear decomposition starting from $C$, $G[U]=C+P''_1+P''_2+\dots+P''_{k''}$.  We use the construction that picks the first ear that includes the edge $(v,r)$ for some $r\neq i,j$ (similar to case 1b). We know that such a node exist, since we assumed $v$ has a degree at least 3 in $G[U]$. Now, we know (again from theorem~\ref{thm:fcNC}) that $C+P''_1$ is a nice subgraph of $G$. So we get a p.m in $G[U]-C-P''_1$. The ear, $P''_1$, (whether proper or not) is still odd and has at least two new nodes (that are not in $C$). So we can delete one of the interior nodes to get a near p.m. with $v$ being matched to one of the interior nodes in $P''_1$ and both $i$ and $j$ being matched to $t\in I$ and $t'\in J$ respectively in the cycle $C$. Thus we have obtained a near p.m for $G[U]$ with both $i$ and $j$ not matched within $G[U']$ and one of the incident edges of $v$ picked in the matching.

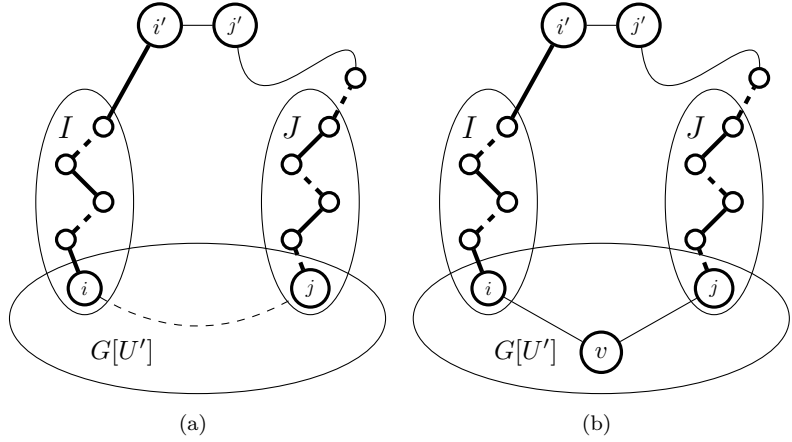
\begin{figure}
\begin{center}
         \subfigure[	]{
       \begin{tikzpicture} [>=stealth]
	\node (i) at (-0.5,0.85)   [scale=0.75,circle,draw,very thick] {$i$};	
	\node (j) at (2.5,0.85)   [scale=0.75,circle,draw,very thick] {$j$};	
	\node (i') at (0.5,4.35)   [scale=0.85,circle,draw,very thick] {$i'$};	
	%\node (v) at (1,0)   [scale=0.95,circle,draw,very thick] {$v$};	
	\node (j') at (1.5,4.35)   [scale=0.75,circle,draw,very thick] {$j'$};	
	\node(1) at (-0.75,1.5) [scale=0.75,circle,draw,very thick] {};
	\node(2) at (-0.25,2) [scale=0.75,circle,draw,very thick] {};
	\node(3) at (-0.75,2.5) [scale=0.75,circle,draw,very thick] {};
	\node(4) at (-0.25,3) [scale=0.75,circle,draw,very thick] {};
	\node(5) at (2.25,1.5) [scale=0.75,circle,draw,very thick] {};
	\node(6) at (2.75,2) [scale=0.75,circle,draw,very thick] {};
	\node(7) at (2.25,2.5) [scale=0.75,circle,draw,very thick] {};
	\node(8) at (2.75,3) [scale=0.75,circle,draw,very thick] {};
	\node(9) at (3.1,3.65) [scale=0.75,circle,draw,very thick] {};

	\draw (1,0.45) ellipse (2.5cm and 1cm);
	\draw (-0.5,2) ellipse (0.65cm and 1.5cm);
	\draw (2.5,2) ellipse (0.65cm and 1.5cm);
	\draw (i')--(j');
 	\node (0,-2)[draw=none,fill=none] {\textbf{\text{$G[U']$}}}; 

	\draw[ultra thick](i) -- (1);
	\draw[ultra thick,dashed](1) -- (2);
	\draw[ultra thick](2) -- (3);
	\draw[ultra thick,dashed](3) -- (4);
	\draw[ultra thick](i') -- (4);
	\draw[ultra thick,dashed](j) -- (5);
	\draw[ultra thick](5) -- (6);
	\draw[ultra thick,dashed](6) -- (7);
	\draw[ultra thick](7) -- (8);
	\draw[ultra thick, dashed](8) -- (9);
	%\draw[decoration = {zigzag,segment length = 5mm, amplitude = 3mm}, decorate](j') -- (9);
 	\node (I) at (-0.75,3)[draw=none,fill=none] {\large\textbf{\text{$I$}}}; 
 	\node (J) at (2.25,3)[draw=none,fill=none] {\large\textbf{\text{$J$}}};	
	
    \draw (j') .. controls (1.7,2.667) and (3.134,4.667) .. (9);
	
	\draw[dashed] (i) edge [bend right] (j);			
        \end{tikzpicture}
        \label{fig:case3ca}
	}
	         \subfigure[]{
        \begin{tikzpicture} [>=stealth]
	\node (i) at (-0.5,0.85)   [scale=0.75,circle,draw,very thick] {$i$};	
	\node (j) at (2.5,0.85)   [scale=0.75,circle,draw,very thick] {$j$};	
	\node (i') at (0.5,4.35)   [scale=0.85,circle,draw,very thick] {$i'$};	
	\node (v) at (1,0)   [scale=0.95,circle,draw,very thick] {$v$};	
	\node (j') at (1.5,4.35)   [scale=0.75,circle,draw,very thick] {$j'$};	
	\node(1) at (-0.75,1.5) [scale=0.75,circle,draw,very thick] {};
	\node(2) at (-0.25,2) [scale=0.75,circle,draw,very thick] {};
	\node(3) at (-0.75,2.5) [scale=0.75,circle,draw,very thick] {};
	\node(4) at (-0.25,3) [scale=0.75,circle,draw,very thick] {};
	\node(5) at (2.25,1.5) [scale=0.75,circle,draw,very thick] {};
	\node(6) at (2.75,2) [scale=0.75,circle,draw,very thick] {};
	\node(7) at (2.25,2.5) [scale=0.75,circle,draw,very thick] {};
	\node(8) at (2.75,3) [scale=0.75,circle,draw,very thick] {};
	\node(9) at (3.1,3.65) [scale=0.75,circle,draw,very thick] {};

	\draw (1,0.45) ellipse (2.5cm and 1cm);
	\draw (-0.5,2) ellipse (0.65cm and 1.5cm);
	\draw (2.5,2) ellipse (0.65cm and 1.5cm);
	\draw (i')--(j');
 	\node (0,-2)[draw=none,fill=none] {\textbf{\text{$G[U']$}}}; 

	\draw[ultra thick](i) -- (1);
	\draw[ultra thick,dashed](1) -- (2);
	\draw[ultra thick](2) -- (3);
	\draw[ultra thick,dashed](3) -- (4);
	\draw[ultra thick](i') -- (4);
	\draw[ultra thick,dashed](j) -- (5);
	\draw[ultra thick](5) -- (6);
	\draw[ultra thick,dashed](6) -- (7);
	\draw[ultra thick](7) -- (8);
	\draw[ultra thick, dashed](8) -- (9);
	%\draw[decoration = {zigzag,segment length = 5mm, amplitude = 3mm}, decorate](j') -- (9);
	
    \draw (j') .. controls (1.7,2.667) and (3.134,4.667) .. (9);
	
	\draw[-] (i) -- (v);	
	\draw[-] (j) -- (v);	
 	\node (I) at (-0.75,3)[draw=none,fill=none] {\large\textbf{\text{$I$}}}; 
 	\node (J) at (2.25,3)[draw=none,fill=none] {\large\textbf{\text{$J$}}};	
	
        \end{tikzpicture}
        \label{fig:case3cc}
	}
\end{center}
\caption{Constructing the nice odd cycle $C$ in $G_\ell$ in the case 3c.}
\end{figure}

\textbf{Case 4:} If $U$ is an odd cycle with no chords, we just pick some node $v\in U$ as our candidate. For any type (i) inequality $e\in E$, we can pick this edge in our matching. If $e=(i,j)\notin E[U]$, it is easy to see. If one of the endpoints is in $U$, we get the near perfect matching in $G[U]$ with the deletion of this node and pick this edge. If $e \in U$, then we can pick this edge and some edge incident to $v$, if $U$ is not a triangle with both $i\neq v$ and $j\neq v$. In the latter case, we know $v$ has a degree 3 in $G$, else $G[U]$ is a rank 0 inequality from lemma~\ref{lem:rank} ($i,j, v$ form a triangle with $v$ having a degree 2). So we can pick this edge that is incident on $v$.

All other cases, follow the same arguments. We won't have the subcase $U'\subset U$, as any $U' \subset U$ will not be factor critical because $U$ is an odd cycle without chords.
\end{proof}
\bibliographystyle{plain}
\bibliography{geometric_rank}

\end{document}